\documentclass[11pt]{amsart}
\usepackage{amsmath,amssymb,amsthm}
\usepackage[english]{babel}
\usepackage{csquotes}

\def \p{\varphi}

\def \p {\partial}

\def\O {\Omega}

\newtheorem{theorem}{Theorem}[section]

\newtheorem{remark}{Remark}[section]

\author{Andres Contreras}
\address{Science Hall 224, New Mexico State university, Department of Mathematical Sciences}
\email{acontre@nmsu.edu}
\author{Xavier Lamy}
\address{Max Planck Institute for Mathematics in the Sciences, Leipzig, Germany}
\email{xlamy@mis.mpg.de}
\author{R\'emy Rodiac}
\address{Universit\'e Paris-Est Cr\'eteil, 61 Avenue du g\'en\'eral de Gaulle, 94010 Cr\'eteil Cedex}
\email{remy.rodiac@u-pec.fr}

\title{Boundary regularity of weakly anchored harmonic maps}

\begin{document}

\begin{abstract}
In this note we study the boundary regularity of minimizers of a family of weak anchoring energies that model the states of liquid crystals. We establish optimal boundary regularity in all dimensions $n\geq 3 .$ In dimension $n=3,$ this yields full regularity at the boundary which stands in sharp contrast with the observation of boundary defects in physics works. We also show that, in the cases of weak and strong anchoring, regularity of minimizers is inherited from that of their corresponding limit problems.The analysis rests in a crucial manner on the fact that the surface and Dirichlet energies scale differently; we take advantage of this fact to reduce the problem to the known regularity of tangent maps with zero Neumann conditions.
\end{abstract}
\maketitle
\section{Introduction}
 Let $n\geq 3,$  $\Omega\subset\mathbb{R}^n$ a smooth bounded domain and $\mathcal{N}$ a smooth compact 
manifold. We are interested in the boundary regularity of minimizers of the family of weak anchoring energies defined for maps $u\in H^1(\Omega; \mathcal{N})$,
\begin{equation} \label{wae} 
E_{w}(u):=\int_{\Omega}\vert\nabla u\vert^2 dx + w\int_{\partial \Omega} g(x,u) d\mathcal{H}^{n-1},
\end{equation}
that arise in the study of liquid crystals \cite{klemanlavrentovich07}.    Above, $w\geq 0$ is referred to as the \textit{anchoring strength} while $g$, the \textit{anchoring energy density}, is a non-negative bounded function on $\partial\Omega\times\mathcal N$. The Euler-Lagrange equations satisfied by a minimizer $u$ of $E_w$ are
\begin{equation}\label{E.L}
\left\lbrace
\begin{array}{rcll}
-\Delta u & = A_{\mathcal{N}}(u)(\nabla u,\nabla u) & \quad \text{in }&\Omega,\\
\frac{1}{w} \frac{\partial u}{\partial \nu} & =\pi_{\mathcal{N}}(u) \nabla_{u} g (x,u)& \quad \text{on }&\partial\Omega,
\end{array}
\right.
\end{equation}
where $\nu$ is the outward unit normal to $\p \O$, $A_{\mathcal{N}}$ is the second fundamental form of $\mathcal{N}$ and $\pi_{\mathcal{N}}$ is the projection on the tangent space.  In the context of liquid crystals,  $n=3$ and the target manifold is $\mathcal N=\mathbb S^2$. The functional $E_w$ relaxes the physically unrealistic strong anchoring constraint
\begin{equation}\label{stronganchoring}
g(x,u(x))=0\quad\text{for a.e. }x\in\partial\Omega,
\end{equation}
which formally corresponds to $w=\infty$. A model case of anchoring density, though not the only one of physical interest,  is given by 
\begin{equation} \label{modelg} 
g(x,u)=\vert u -u_0(x)\vert^2,
\end{equation}
for some $u_0\colon\partial\Omega\to\mathcal N,$
%\begin{equation}\label{modelg}
%g(x,u)=\vert u -u_0(x)\vert^2,\quad u_0\colon\partial\Omega\to\mathcal N,
%\end{equation}
which corresponds to Dirichlet boundary conditions in the strong anchoring limit.

Interior regularity for minimizers of $E_w$ follows directly from \cite{SUregularity}: the singular set has Hausdorff dimension at most $n-3$ and is discrete when $n=3$. On the other hand, boundary regularity does not seem to have been considered. In the liquid crystal setting, however, boundary defects have been discussed in \cite{klemanlavrentovich07}. The chief goal of this note is to address the question of optimal boundary regularity of minimizers of $E_w.$ We  tackle this question from two different perspectives: first we obtain an optimal bound on the dimension of the singular set of such maps valid for all values of $w,$ and then we take on a perturbation point of view to observe that boundary smoothness is a stable condition in $w.$

In what follows $\mathrm{Sing}(u)\subset \overline\Omega$ denotes the set of points where $u$ is not continuous while $\dim A$ corresponds to the Hausdorff dimension of a set $A\subset\mathbb R^n.$  As is natural, we extend the definition of $E_w$ to $w=\infty$ by setting $E_\infty(u)=+\infty$ if $u$ does not satisfy the strong anchoring constraint \eqref{stronganchoring}. Our main result is the following.
\begin{theorem}\label{A} Let $E_w$ be as in \eqref{wae}. The following holds about minimizers of $E_w$ in $H^1(\Omega;\mathcal N):$
\begin{itemize}
\item[] {\bf 1.} (Optimal boundary regularity for fixed anchoring strength)
 For any $w\in [0,\infty)$ and $u$ a minimizer of $E_w,$ 
\begin{equation*}
\begin{aligned}
&\dim (\mathrm{Sing}(u)\cap\partial\Omega)\leq n-4 &\quad \text{if }n\geq 4,\\
&\mathrm{Sing}(u)\cap\partial\Omega \text{ is discrete }&\quad\text{if }n=4,\\
&\mathrm{Sing}(u)\cap\partial\Omega=\emptyset &\quad\text{if }n=3.
\end{aligned}
\end{equation*}
\item[] {\bf 2.} (Stability with respect to the anchoring strength)
 Assume that for some $w_0\in [0,\infty]$, minimizers of $E_{w_0}$ have no boundary singularities, and in the case $w_0=\infty$ assume in addition that $\inf E_{w_0} <\infty$. Then, for $w$ in a neighborhood of $w_0$, minimizers of $E_w$ have no boundary singularities.
\end{itemize}
\end{theorem}

\hspace{0.5cm}

Let us note that, somewhat surprisingly, in the case of the physical dimension $n=3,$ the first part of Theorem~\ref{A} gives full regularity at the boundary which is in strong contrast with physical observations \cite{klemanlavrentovich07}. At the same time, the second part of the theorem implies in particular that minimizers of $E_w$ have no boundary singularities for $w$ close to zero (weak anchoring case), 
since minimizers of the Dirichlet energy with Neumann boundary conditions are constants.  In the case of extreme anchoring (that is when $w$ is large)  and for $g$ of the form \eqref{modelg} with a smooth $u_0$, minimizers of $E_w$ have no boundary singularities, again as a consequence of the second part of Theorem \ref{A} because minimizing harmonic maps with smooth Dirichlet conditions are smooth near the boundary \cite{SUboundaryregularity}.

%\begin{theorem}\label{wae}
%Let $n\geq 3,$ $w>0$ and $u_w$ a minimizer of $E_w.$ We have
%\begin{itemize}
%\item[a)] For all $w>0$ it holds that $dim\left(Sing(u_w)\cap \partial \Omega\right)\leq n-4.$ 
%\item[b)] The set $\mathcal{A}:=\{w>0: dim(Sing(u)\cap\partial \Omega)=0\}$ is an open set containing $0$ and $+\infty.$
%\end{itemize}
%\end{theorem}

\section{Proof of Theorem \ref{A}}

The proof of the first part of Theorem \ref{A} follows the classical scheme for regularity of harmonic maps \cite{SUregularity}, which relies on the study of tangent maps.
Let $x_0\in \partial\Omega$ and $r>0.$ Defining  $\hat{u}(x):=u(x_0+rx),$ we have
 \begin{equation*}
\left\lbrace
\begin{aligned}
-\Delta \hat{u} & = A_{\mathcal{N}}( \hat{u})(\nabla  \hat{u},\nabla  \hat{u}) & \quad\text{in }&\frac{1}{r}\left( \O \setminus{x_0} \right),\\
\frac{1}{w} \frac{\partial  \hat{u}}{\partial \nu} & =r \pi_{\mathcal{N}}( \hat{u}) \nabla_{ \hat{u}} g (x, \hat{u})& \quad\text{on }&\partial \left[\frac{1}{r}\left( \O \setminus{x_0} \right) \right],
\end{aligned}
\right.
\end{equation*}
%where $\nu$ is the outward unit normal to $\p \O$, $A_{\mathcal{N}}$ is the second fundamental form of $\mathcal{N}$ and $\pi_{\mathcal{N}}$ is the projection on the tangent space. 
Since $\pi_{\mathcal{N}}( \hat{u}) \nabla_{ \hat{u}} g (x, \hat{u})$ is bounded, taking the formal limit  $r \rightarrow 0$ yields a map $\phi$ satisfying
\begin{equation*}
\left\lbrace
\begin{aligned}
-\Delta \phi & = A_{\mathcal{N}}(\phi)(\nabla \phi,\nabla \phi) & \quad\text{in }&\mathbb{R}^n_+,\\
 \frac{\partial \phi}{\partial \nu} & = 0& \quad\text{on }& \mathbb{R}^{n-1}\times \{0\}.
\end{aligned}
\right.
\end{equation*}
 Such maps, when they are 
 $0$-homogeneous and locally minimizing, are called \textit{free-boundary minimizing tangent maps} and have been studied by Hardt and Lin in \cite{HLpartially}. 
 They discovered that their singular set has Hausdorff dimension at most $n-4$ at the boundary. 
 This result allows to conclude, provided we adapt the techniques developed in \cite{SUregularity} to our case. 
 An essential ingredient in \cite{SUregularity} is the energy monotonicity formula, which -- together with a technical extension lemma -- ensures convergence of  blow-up sequences to tangent maps.  The key observation in our case is that the surface anchoring term in the energy \eqref{wae} scales differently from the Dirichlet energy whence an approximate monotonicity formula is still valid; moreover the surface anchoring term disappears after blow-up and thus our tangent maps are precisely the ones studied in \cite{HLpartially}, where the equivalent of Theorem~\ref{A} part 1. is proven.

\begin{proof}[Proof of Theorem~\ref{A} part 1] We denote by $B_r^+$ the half ball
\begin{equation*}
B_r^+ =\left\lbrace x\in\mathbb R^n \colon \vert x\vert<1,\, x_n>0\right\rbrace,
\end{equation*}
by $\Sigma_r$ the \enquote{flat} part of its boundary $\Sigma_r =  B_r^+\cap\lbrace x_n=0\rbrace$, and by $\Gamma_r$ the \enquote{round} part of its boundary $\Gamma_r=\partial B_r^+ \cap \lbrace x_n>0\rbrace$.
By locally flattening the boundary of $\Omega$, our problem can be reduced to studying maps minimizing an energy functional of the form
\begin{equation*}
\mathcal E_w(u)=\int_{B_1^+}\vert\nabla u\vert^2 +w\int_{\Sigma_1} g(x',u),
\end{equation*}
among maps $u\in H^1(B_1^+;\mathcal N)$ with fixed boundary values on $\Gamma_r$. Here $g$ is a bounded non-negative function on $\Sigma_1\times \mathcal N$, and we study regularity of minimizers on $\Sigma_1$. It is important to remark that the two terms in the above energy scale differently: setting $u_r(x)=u(rx)$, it holds
\begin{equation}\label{scale}
r^{n-2}\mathcal E_w(u_r)=\int_{B_r^+} \vert \nabla u \vert^2 + r \int_{\Sigma_r} g(x'/r,u).
\end{equation}
A first consequence of \eqref{scale} is that \enquote{small energy regularity} holds for minimizers of $\mathcal E_w$: there exist $r_0$ and $\varepsilon_0$ (depending on $n$, $w$ and $\sup g$) such that for $r<r_0$,
\begin{equation}\label{smallreg}
r^{2-n}\int_{B_r^+} \vert \nabla u\vert^2 <\varepsilon_0^2 \quad \Longrightarrow \quad u\text{ is continuous in }\overline {B_{r/2}^+}.
\end{equation}
This can be proved arguing by contradiction as in \cite[Proposition~1]{luckhaus88}. An essential step there is to construct good comparison maps, which is done with the help of an important extension lemma. Our setting is slightly different since we are dealing with maps defined on half balls, but after extending by reflection, the proof carries over. 

Comparison with rescaled homogeneous maps as in 
\cite[\S 2]{SUregularity} implies the following monotonicity formula: for some $c>0$ depending only on $n$, $w$ and $\sup g$,
\begin{equation}\label{monot}
\frac{d}{dr}\left[r^{2-n}\int_{B_r^+}\vert\nabla u\vert^2 \right]\geq -c +  r^{2-n}\int_{\Gamma_r}\left\vert \frac{\partial u}{\partial r}\right\vert^2.
\end{equation}
Together with the construction of good comparison maps, this monotonicity formula implies, following \cite[Proposition~2]{luckhaus88} and taking \eqref{scale} into account, the strong $H^1$ convergence of blow-up subsequences $u_{x_0,r_i}(x)=u(x_0+ r_i x)$ for any $x_0\in\Sigma_1$. The limits, called \textit{tangent maps}, are homogeneous $\mathcal N$-valued maps  defined in the half-space $\lbrace x_n>0\rbrace$. Moreover tangent maps minimize the Dirichlet energy $\mathcal E_0$ with \emph{free boundary conditions} on 
 $\Sigma_1$ (and fixed boundary values on $\Gamma_1$). Therefore the proof can be continued exactly as in \cite[Theorem~2.8]{HLpartially}.
\end{proof}

\begin{remark}\label{rem:density}
The Dirichlet energy of a tangent map at $x_0\in\overline\Omega$ equals the \emph{density function}
\begin{equation}\label{theta}
\Theta(u,x_0)=\lim_{r\to 0} \left[ r^{2-n}\!\int_{B_r(x_0)\cap\Omega}\vert\nabla u\vert^2\right].
\end{equation}
The limit exists thanks to the monotonicity formula \eqref{monot}. The small energy regularity property \eqref{smallreg} amounts to
\begin{equation}\label{smalltheta}
\Theta(u,x_0)<\varepsilon_0 \quad\Longrightarrow\quad u\text{ is continuous at }x_0,
\end{equation}
and $\varepsilon_0$  can be \emph{a posteriori} taken as the infimum of the Dirichlet energy over all non constant tangent maps. In particular, $\varepsilon_0$ in \eqref{smalltheta} is \emph{independent} of $w$ and $\sup g$, which was \emph{a priori} not obvious.
\end{remark}
Theorem~\ref{A} part 2. follows from the strong $H^1$ convergence of minimizers of $E_w$ to minimizers of $E_{w_0}$ as $w\to w_0$.
\begin{proof}[Proof of Theorem~\ref{A} part 2] Were the result not true, there would exist a sequence $w_k\to w_0$ and maps $u_k$ minimizing $E_{w_k}$, with singularities at $x_k\to x_0\in\partial\Omega$. By Remark~\ref{rem:density} above, this implies in particular $\Theta(u_k,x_k)\geq \varepsilon_0$. 

We may also assume that $u_k$ converges weakly in $H^1$ to a $\mathcal N$-valued map $u_0$. The convergence is in fact strong, and  $u_0$ minimizes $E_{w_0}$: this follows from the inequalities
\begin{equation*}
E_{w_0}(u_0)\leq \liminf E_{w_k}(u_k) \leq \liminf E_{w_k}(u) = E_{w_0}(u),\qquad\forall u\in H^1(\Omega ;\mathcal N).
\end{equation*}
By assumption, $u_0$ has no boundary singularities. On the other hand it holds $\Theta(u_0, x_0)\geq \varepsilon_0$ since the density function is upper-semicontinuous \cite[Proposition 10.26]{GMregularitytheory}. This contradiction completes the proof.
\end{proof}
\section{Future directions}

The proof of Theorem \ref{A} is really specific to minimizing maps. A natural question is then if it can be extended to also consider stationary harmonic maps. Another line of investigation is more directly linked to the  harmonic map depiction of liquid crystals, which can be seen as the London limit of a more general model based on  $Q$-tensors \cite{majumdarzarnescu10}: in the case of weak anchoring, does the convergence of minimizing $Q$-tensors hold up to the boundary? 
Finally, the upper bound in Theorem \ref{A} part {\bf 1.} is very general and valid for any $w$ and any function $g$. It would be interesting to see if this bound can be improved incorporating the dependence on the anchoring strength and the map $g$. This would require much finer analysis.

% etc, etc

% The Appendices part is started with the command \appendix;
% appendix sections are then done as normal sections
% \appendix

% \section{}
% \label{}

% The Acknowledgements are an un-numbered section
\section*{Acknowledgements}
A. Contreras and R. Rodiac would like to thank the Institut Camille Jordan, Universit\'e de Lyon-1, where this work was carried out, for the hospitality and a great atmosphere. 

\bibliographystyle{plain}
\bibliography{biblioregharmonic}

\end{document}